\documentclass[10pt,a4paper]{article}
\usepackage[affil-it]{authblk}
\usepackage{amsmath,amsthm,amssymb,cite}
\usepackage{lineno,hyperref}
\usepackage{enumerate}
\usepackage[mathscr]{euscript}

\usepackage{bm}

\usepackage{calrsfs}

\newcommand{\bTop}{\operatorname{\bf Top}}

\newcommand{\bY}{\boldsymbol{Y}}
\newcommand{\bZ}{\boldsymbol{Z}}

\newcommand{\bX}{\operatorname{\bf X}}

\newcommand{\kw}{\operatorname{\mathcal{W}}}
\newcommand{\ka}{\operatorname{\mathcal{A}}}
\newcommand{\kb}{\operatorname{\mathcal{B}}}
\newcommand{\kv}{\operatorname{\mathcal{V}}}

\newcommand{\ku}{\operatorname{\mathcal{U}}}
\newcommand{\ki}{\operatorname{\mathcal{I}}}

\newcommand{\M}{\operatorname{\bf M}}

\newcommand{\Bo}{\operatorname{B_0}}
\newcommand{\Bci}{\operatorname{[B_4]}}
\newcommand{\bBo}{\operatorname{\bf B_0}}

\newcommand{\bq}{\operatorname{\bf q}}

\newcommand{\ANR}{\operatorname{ANR}}
\newcommand{\AR}{\operatorname{AR}}

\newcommand{\Ss}{\operatorname{SS}}
\newcommand{\rel}{\text{rel}}
\newcommand{\ANRF}{\operatorname{A(N)R}}
\newcommand{\ssh}{\operatorname{SSH}}
\newcommand{\bssh}{\operatorname{\bf SSH}}

\newcommand{\Ae}{\operatorname{AE}}
\newcommand{\ANE}{\operatorname{ANE}}
\newcommand{\ANEF}{\operatorname{A(N)E}}

\newcommand{\HS}{\operatorname{H}}

\newcommand{\bpro}{\operatorname{\bf pro}}
\newcommand{\F}{\operatorname{F}}
\newcommand{\bsh}{\operatorname{\bf SH}}
\newcommand{\SSS}{\operatorname{S}}

\newcommand{\kc}{\operatorname{\mathcal{C}}}
\newcommand{\bW}{\operatorname{\bf W}}
\newcommand{\bp}{\operatorname{\bf p}}
\newcommand{\br}{\operatorname{\bf r}}

\newcommand{\bff}{\operatorname{\bf f}}
\newcommand{\bgg}{\operatorname{\bf g}}
\newcommand{\C}{\operatorname{C}}
\newcommand{\E}{\operatorname{E}}
\newcommand{\HH}{\operatorname{H}}
\newcommand{\sh}{\operatorname{sh}}
\newcommand{\ssshh}{\operatorname{ssh}}
\newcommand{\pssh}{\operatorname{ssh}}
\newcommand{\SE}{\operatorname{SE}}
\newcommand{\const}{\operatorname{const}}

\usepackage{tikz}
\usetikzlibrary{arrows,chains,matrix,positioning,scopes}
\makeatletter
\tikzset{join/.code=\tikzset{after node path={%
\ifx\tikzchainprevious\pgfutil@empty\else(\tikzchainprevious)%
edge[every join]#1(\tikzchaincurrent)\fi}}}
\makeatother
\tikzset{>=stealth',every on chain/.append style={join},
         every join/.style={->}}

\newtheorem{theorem}{Theorem}[section]
\newtheorem{lemma}[theorem]{Lemma}
\newtheorem{corollary}[theorem]{Corollary}
\theoremstyle{definition}
\newtheorem{remark}[theorem]{Remark}
\newtheorem{Definition}[theorem]{Definition}

\newtheorem{proposition}[theorem]{Proposition}

\title{Fiber Strong Shape Theory for Topological Spaces{\footnote{The authors supported in part by grant FR/233/5-103/14 from Shota Rustaveli National Science Foundation (SRNSF)}}}
\date{\vspace{-5ex}}

\author{Vladimer Baladze and Ruslan Tsinaridze}
\affil{Department of Mathematics of\\ Shota Rustaveli Batumi State University}
\providecommand{\keywords}[1]{\textbf{\textit{Keywords and Phrases:}} #1}
\begin{document}
\maketitle

\setcounter{section}{-1}

\begin{abstract}
In the paper we construct and develop a fiber strong shape theory for arbitrary spaces over fixed metrizable space $\Bo$. Our approach is based on the method of Marde\v{s}i\'{c}-Lisica and instead of resolutions, introduced by Marde\v{s}i\'{c}, their fiber preserving analogues are used. The fiber strong shape theory yields the classification of spaces over $\Bo$ which is coarser than the classification of spaces over $\Bo$ induced by fiber homotopy theory, but is finer than the classification of spaces over $\Bo$ given by usual fiber shape theory.
\end{abstract}

\textbf{Math. Sub. Class.:}54C55, 54C56, 55P55.

\keywords{Fiber shape, Fiber homotopy, Fiber resolution, Fiber shape expansion, Fiber strong expansion, $\ANRF_{\Bo}$-space, $\ANEF_{\Bo}$-space.}

\section{Introduction}
Together with the classical shape theory and its various versions  there exists an important branch of modern geometric topology, the so called strong shape theory, which besides the applications in topology (general topology, algebraic topology, geometric topology) \cite{80}, has also applications in other branches of mathematics (dynamical systems, C$^*$-algebras)(\cite{64},\cite{45}). 

Strong shape theory for different categories of spaces was investigated by several authors. For the category of compact metric spaces equivalent strong shape theories were introduced by F.W.Bauer \cite{25}, A. Calder and H.M.Hastings \cite{38}, F.W.Cathey (\cite{32},\cite{33}), J.Dydak and J.Segal \cite{54}, D.A.Edwards and H.M.Hastings \cite{56}, Y.Kodama and J.Ono \cite{73}, Yu.T.Lisica \cite{77} and J.B.Quigly \cite{96}.

Strong shape theory for the category of general topological spaces and arbitrary categories was constructed by M. Batanin \cite{24}, F.W. Bauer \cite{25}, J.Dydak and S.Nowak (\cite{52}, \cite{53}), Yu.T.Lisica \cite{76}, Yu.T.Lisica and S.Marde\v{s}ic' \cite{78}, Z.Miminoshvili \cite{85} and L. Stramaccia \cite{102}. 

For the present period of the shape theory development it is characteristic to design and research different versions of strong shape theory.

Strong shape theory based on the notion of equivariant homotopy constructed by V.Baladze \cite{8} for metric $G$-spaces and A.Bykov and M.Texis for compact metric $G$-spaces \cite{31}.

Strong shape theory based on the notion of $n$-homotopy was developed by Y.Iwamoto and K.Sakai \cite{66}. 

The problem of construction of strong shape theory for fiberwise topology is one of the interesting problems because fiberwise topology occupies a central place in topology today. As the strong shape theory arises from homotopy theory, so fiber strong shape theory arises from fiberwise homotopy theory.

To develop the fiber strong shape theory is natural. It is hoped that this may stimulate further research of fiberwise topology, in particular, of fiberwise homotopy theory.

The main aim of the paper is to develop fiber strong shape theory for arbitrary spaces over fixed metrizable space $\Bo$. Marde\v{s}i\'{c}-Lisica approach (\cite{78},\cite{80}) to strong shape theory can in fact be extended to the fiber version. The construction of fiber strong shape category is based on the general method developed by Ju.T.Lisica and Marde\v{s}i\'{c} (\cite{78},\cite{80} and uses the notions of fiber $\ANR_{\Bo}$-resolutions and fiber strong $\ANR_{\Bo}$-expansions.

Fiber strong expansions of spaces over $\Bo$ are morphisms of category $\bpro-\bTop_{\bBo}$ from spaces over $\Bo$ to inverse systems of spaces over $\Bo$, which satisfy a stronger version of fiber homotopy conditions of $\ANR_{\Bo}$-expansion defined by V.Baladze (\cite{11}-\cite{16}).

In the paper it is proved that fiber resolutions of spaces over $\Bo$ induce fiber strong expansions of spaces over $\Bo$. In order to construct the fiber strong shape category $\ssh_{\Bo}$ we use this result. 

Besides, here also are proved that there exist fiber strong shape functor $\Ss:\HS(\bTop_{\bBo})\to \bssh_{\bBo}$ and functor $\F:\bssh_{\Bo}\to \bsh_{\bBo}$ such that $\F \circ \Ss=\SSS$, where $\SSS:\HS(\bTop_{\bBo})\to \bsh_{\bBo}$ is fiber shape functor.

Now give some notation and preliminaries.

We use the following notations. Let $\Bo$ denote the fixed space. The space $X$ over $\Bo$ is a pair consisting of a topological space $X$ and a continuous mapping $\pi_X:X\to \Bo$. Let $X$ and $Y$ be spaces over $\Bo$. A continuous map $f:X\to Y$ is said to be a fiber preserving (f.p.) if $\pi_Y ~ f=\pi_X$. By $\bTop_{\bBo}$ we denote the category of all spaces over $\Bo$ and all f.p. maps.

Two f.p. maps $f,g:X\to Y$ of the category $\bTop_{\bBo}$ are said to the fiber preserving (f.p.) homotopic, $f \mathop {\simeq}\limits_{\Bo} g$, if there is a homotopy $H:X\times I\to Y$ from $f$ to $g$, such that $\pi_Y ~ H=\pi_{X\times I}$, where $\pi_{X\times I}(x,t)= \pi_{X}(x)$ for every $t\in I$ and $ x \in X$. The homotopy $H$ is called fiber preserving homotopy or homotopy over $\Bo$. The relation $\mathop {\simeq}\limits_{\Bo}$ is an equivalence relation. We denote by $[f]_{\Bo}$ the homotopy class of the f.p. map f. The relation $\mathop {\simeq}\limits_{\Bo}$ is compatible with the composition. Therefore, one can define the composition of classes $[f]_{\Bo}:X\to Y$ and $[g]_{\Bo}:Y\to Z$ by composing representatives: 

\[
[g]_{\Bo} ~ [f]_{\Bo}=[g ~ f]_{\Bo}.
\]

$\HH(\bTop_{\bBo})$ denotes homotopy category of $\bTop_{\bBo}$. Its objects are all objects of $\bTop_{\bBo}$ and the morphisms are equivalence classes with respect to relation $\mathop {\simeq}\limits_{\Bo}$ of morphisms in $\bTop_{\bBo}$. Two spaces over $\Bo$, X and Y are said to be fiber homotopy equivalent, $X\mathop {\simeq}\limits_{\Bo} Y$, if there exist two f. p. maps $f:X\to Y$ and $g:Y\to X$ such that $g ~ f \mathop {\simeq}\limits_{\Bo} 1_X$ and $f ~ g \mathop {\simeq}\limits_{\Bo} 1_Y$.

We denote by $C(I,X)$ the space of all continuous maps $\varphi: I\to X$ endowed with the compact-open topology. $C_{\Bo}(I,X)$ denotes the subspace of $C(I,X)$ consisting of all continuous maps $\varphi:I\to X$ such that $\pi_X ~ \varphi=\const$.

Let $\Bo$ be a fixed metrizable space and $\M_{\Bo}$ the category of all metrizable spaces over $\Bo$ and all f.p. maps.

Let $X$ be a metrizable space over $\Bo$ and let $Y$ be a subspace of $X$. A f.p. map $r:X\to Y$ is called a fiberwise retraction if $r ~ i=1_Y$, where $i:Y \to Y$ is the f.p. inclusion map. In this case the subspace Y is called a fiber retract of $X$.

A subspace $Y$ of metrizable space $X$ over $\Bo$ is called a fibrewise neighbourhood retract of $X$ if there exist a neighbourhood $U$ of $Y$ in $X$ and a fibrewise retraction $r:U\to Y$.

The space $Y\in \M_{\Bo}$ is an absolute retract over $\Bo$, $Y\in \AR_{\Bo}$ (an absolute neighbourhood retract over $\Bo$, $Y\in \ANR_{\Bo}$), if $Y$ has the following property: for any closed f.p. embedding $i:Y\to X \in \M_{\Bo}$ there exists fibrewise retraction $r:X\to i(Y)$ (there exist a neighbourhood $U$ of $i(Y)$ in $X$ and a fibrewise retraction $r:U\to i(Y)$).

The space $Y \in \M_{\Bo}$ is an absolute extensor over $\Bo$, $Y \in \Ae_{\Bo}$ (an absolute neighbourhood extensor over $\Bo$, $Y \in \ANE_{\Bo}$), if it has the following property: for any space $X \in \M_{\Bo}$ and any closed subset $A \subseteq X$, every f.p. map $f:X\to Y$ admits a f.p. extension $\underline{f}:X\to Y$ ($\underline{f}:U\to Y$, where $U$ is a neighbourhood of $A$ in $X$) (see \cite{91}).

\begin{proposition}\label{proposition 0.1}
\textit{ A space $Y$ over $\Bo$ is an $\ANR_{\Bo}$-space if and only if $Y$ is an $\ANE_{\Bo}$-space} \cite{106}.
\end{proposition}
\begin{proposition} \label{proposition 0.2}
\textit{For every metrizable space $X$ over $\Bo$ there exist an $\ANR_{\Bo}$-space $M$ with weight
\[
w(M) \leq max\{w(X),w(\Bo),\aleph_{0} \}
\]
and a f.p. embedding $i:X\to M$ such that $i(X)$ is closed in $M$}(\cite{12},\cite{16}). 
\end{proposition}

Let $\mathscr{U}=\{U_{\alpha}\}_{\alpha\in \ka}$ be a covering of a space $Y$. We say that the maps $f,g:X\to Y$ are $\mathscr{U}$-near, if for every $x \in X$ there exists a $ U_{\alpha}\in \mathscr{U}$ such that, $f(x),g(x) \in U_{\alpha}$. We say that a homotopy $H:X\times I\to Y$ which connects $f$ and $g$, is a $\mathscr{U}$-homotopy if for every $x\in X$ there exists a $U_{\alpha}\in \mathscr{U}$ such that $H(x,t)\subseteq U_{\alpha}$ for all $t\in I$.

\begin{proposition} \label{proposition 0.3}
(Comp. \cite{12},Proposition 7) Let $(Y,\pi_{Y})$ be a ${\rm {ANR}}_{{\rm B_0} }$.Then every open covering $\mathscr{U}$ of $(Y,\pi_{Y})$ admits an open covering $\mathscr{V}$ of $Y$ such that, whenever any two f.p. maps $f,g:(X,\pi_{X})\to (Y,\pi_{Y})$ from an arbitrary space $(X,\pi_{X})$ over $\rm B_0 $ into the space $(Y,\pi_{Y})$ over $\rm B_0 $ are $\mathscr{V}$-near, then there exists f.p. $\mathscr{U}$-homotopy $H:(X\times I,\pi_{X\times I})\to (Y,\pi_{Y})$ which connects $f$ and $g$. Moreover, if for a subset $A\subseteq X$, $f_{|A}=g_{|A}$, then $H$ is f.p. homotopy ${\rm {rel}} A$.
\end{proposition}
\begin{proof} We may assume that $(Y,\pi_{Y})$ is a closed subset of space ${\rm B_0} \times K$, where $K$ is a convex set of normed vector space $L$. Let $\pi: {\rm B_0} \times K\to K$ be the map given by the formula $\pi(b,k)=k$ for every $(b,k)\in {\rm B_0} \times K$. Since $(Y,\pi_{Y})$ is an ${\rm {ANR}}_{{\rm B_0} }$, there is an open neighbourhood $(G,\pi_{G})$ of $(Y,\pi_{Y})$ in ${\rm B_0} \times K $ together with a fibrewise retraction $r:(G,\pi_{G})\to (Y,\pi_{Y})$. Let $\{O_{\alpha}\times Q_{\alpha}\}_{\alpha\in \ka}$ be a refinement of $r^{-1}(\mathscr{U})$, where $Q_{\alpha}$ is convex for every $\alpha \in \ka$. Then $\mathscr{V}=\{(O_{\alpha}\times Q_{\alpha})\cap Y\}_{\alpha\in \ka}$ is an open refinement of the covering $\mathscr{U}$. For any two $\mathscr{V}$-near f.p. maps $f,g:(X,\pi_{X})\to (Y,\pi_{Y}) \subseteq {\rm B_0} \times K $ we can define a f.p. homotopy $H^{‘}:(X\times I,\pi_{X\times I}) \to K$ by formula
\[
H^{'}(x,t)=(\pi_X(x),(1-t)\pi(f(x))+t\pi(g(x)),~~~~~~~(x,t)\in X\times I.
\]

Define a f.p. map $H:(X\times I,\pi_{X\times I})\to (Y,\pi_{Y})$ by taking
\[
H(x,t)=r(H{'}(x,t)),~~~~~~(x,t)\in X\times I.
\]
Clearly, we have $H_0=f$, $H_1=g$ and $H$ is a $\mathscr{U}$-homotopy. Obviously, if $f(x)=g(x)$, for each $x\in A$, then $H(x,t)=f(x)=g(x)$ for every $t\in I$.
\end{proof}

T. Yagasaki \cite{106} showed the following proposition.

\begin{proposition} \label{proposition 0.4}
\textit{Let $Y\in \ANR_{\Bo}$. Let $A$ be a closed subspace of a metrilzable space $X$ over $\Bo$. Let $f,g:X\to Y$ be f.p. maps and let $H:A\times I\to Y$ be f.p. maps and let $H:A\times I\to Y$ be a homotopy over $\Bo$ from $f_{|A}$ to $g_{|A}$. Then there exists a neighbourhood $U$ of $A$ in $X$ and homotopy $\tilde{H}:U\times I\to Y$ over $\Bo$ from $g_{|U}$ to $f_{|U}$ such that $\tilde{H}_{|A\times I}=H$.}
\end{proposition}

Let $C(Z,Y)$ be the function space with the compact-open topology. It is known that if $Z$ is a compact space, then the compact-open topology on $C(Z,Y)$ agrees with the topology induced by the metric:
\[
d(f,g)=sup\{d(f(z),g(z)):z\in Z, f,g\in C(Z,Y)\}.
\]

Consider the subspace $C_{\Bo}(Z,Y)$ of the space $C(Z,Y)$:
\[
C_{\Bo}(Z,Y) =\{f\in C(Z,Y):\pi_Y ~ f = \const\}.
\]

Let $\pi_{C_{\Bo}(Z,Y)}: C_{\Bo}(Z,Y)\to \Bo$ be a map given by $\pi_{C_{\Bo}(Z,Y)}(f)=\pi_Y(f(z)), z\in Z$. Consequently, the pair consisting of the space $C_{\Bo}(Z,Y)$ and the map $\pi_{C_{\Bo}(Z,Y)}$ is a space over $\Bo$.

\begin{proposition} \label{proposition 0.5}
\textit{Let $Y$ be an $\ANR_{\Bo}$-space and let $Z$ be a compact metric space. Then the space $C_{\Bo}(Z,Y)$ is an $\ANR_{\Bo}$-space}(\cite{12},\cite{16}).
\end{proposition}

\section{Resolution and Strong Expansions of Spaces over $\bBo$}
An inverse system of the category $\bTop_{\bBo}$ is a collection  $\bX=(X_{\alpha},p_{\alpha\alpha^{'}},\ka)$ of space $ X_{\alpha}$ over $\Bo$ indexed by a directed set $\ka$ and f.p. maps $ p_{\alpha\alpha^{'}}: X_{\alpha^{'}}\to X_{\alpha}$, $\alpha \leq \alpha^{'}$,  such that $ p_{\alpha\alpha^{'}} ~ p_{\alpha^{'}\alpha^{''}}= p_{\alpha\alpha^{''}}$ and $p_{\alpha\alpha}=1_{X_{\alpha}}$, $\alpha \in \ka$.

A morphism $(f_{\beta},\varphi): \bX \to \bY=(Y_{\beta},q_{\beta\beta^{'}},\kb)$
of inverse systems of the category $\bTop_{\bBo}$ consists of a function $\varphi:\kb\to \ka$ and of f.p. maps $f_{\beta} :X_{\varphi(\beta)}\to Y_{\beta} $, $\beta\in \kb $, such that whenever $\beta \leq \beta^{'}$, then there is an index $\alpha \geq \varphi (\beta), \varphi (\beta^{'})$ for which $f_{\beta} ~ p_{\varphi(\beta)}= q_{\beta\beta^{'}}  ~ f_{\beta^{'}} ~ p_{\varphi(\beta^{'})\alpha}$.

Two morphisms $(f_{\beta},\varphi), (g_{\beta},\psi):\bX\to \bY$ are said to be equivalent, $f\mathop {\simeq}\limits_{\Bo}g$, provided for each $\beta\in \kb$ there is an $\alpha\in \ka$, $\alpha \geq \varphi(\beta), \psi(\beta)$, such that $f_{\beta} ~ p_{\varphi (\beta)\alpha}=g_{\beta} ~ p_{\psi(\beta)\alpha}$.

Let $\bpro-\bTop_{\bBo}$ be a category, whose objects are the inverse systems $\bX$ of the category $\bTop_{\bBo}$ and whose morphisms are the equivalence classes $\bff$ of morphisms $(f_{\beta,\varphi}):\bX\to \bY$ with respect to relation $\mathop {\simeq}\limits_{\Bo}$.

A morphism $\bp=(p_{\alpha}):X\to \bX=(X_{\alpha},p_{\alpha\alpha^{'}},\ka)$ from a rudimentary system $(X)$ to an inverse system $\bX$ consists of the f.p. maps $p_{\alpha}:X\to X_{\alpha}$,$\alpha\in \ka$, such that $p_{\alpha}=p_{\alpha\alpha^{'}} ~ p_{\alpha^{'}}$, $\alpha \leq \alpha^{'}$.

\begin{Definition}\label{Definition1.1}
\textit{Let $X$ be a space over $\Bo$ and let $\bX=(X_{\alpha},p_{\alpha\alpha^{'}},\ka)$ be an inverse system of the category $\bTop_{\bBo}$. We say that $\bp:X\to \bX$ is a resolution over $\Bo$ or fiber resolution of the space $X$ over $\Bo$ provided it satisfies the following two conditions:}

\item[R$_{\Bo}1).$]  \textit{Let $P\in \ANR_{\Bo}$, let $\mathscr{U}$ be an open covering of $P$ and let $h:X\to P$ be a f.p. map. Then there exist an index $\alpha\in \ka$  and a f.p. map $f:X_{\alpha}\to P$  such that $h$ and $f ~ p_{\alpha}$ are $\mathscr{U}$-near.}

\item[R$_{\Bo}2).$] \textit{Let $P\in \ANR_{\Bo}$ and let  $\mathscr{U}$ be an open covering of $P$. Then there is an open cover $\mathscr{U}^{'}$ of $P$ with the following property: if $\alpha\in \ka$ and $f,f^{'}:X\to P$ are f.p. maps such that the f.p. maps $f ~ p_{\alpha}$ and $f^{'} ~ p_{\alpha }$ are $\mathscr{U}^{'}$-near, then there is an index $\alpha^{'} \geq \alpha$ such that the f.p. maps $f ~ p_{\alpha \alpha^{'}}$ and $f^{'} ~ p_{\alpha \alpha^{'}}$ are $\mathscr{U}$-near.}

\item \textit{If in a fiber resolution $\bp:X\to \bX=(X_{\alpha},p_{\alpha\alpha^{'}},\ka)$ of the space $X$ over $\Bo$ each $X_{\alpha}$ is an $\ANR_{\Bo}$, then we say that $\bp$ is a fiber $\ANR_{\Bo}$-resolution.
}\end{Definition}
The next theorem is essential in the construction of the fiber shape category for arbitrary spaces over $\Bo$.

\begin{theorem}\label{theorem 1.2}
Every space $X$ over a metrizable space $\Bo$ admits an $\ANR_{\Bo}$-resolution over $\Bo$.
\end{theorem}

In the proof of Theorem \ref{theorem 1.2} we shall need the following lemma. 
\begin{lemma}\label{lemma 1.3}
Let $f:X\to Y$ be a f.p. map from the topological space $X$ over $\Bo$ to an $\ANR_{\Bo}$-space $Y$. Then there exists an $\ANR_{\Bo}$-space $Z$ of weight
$w(Z)\leq max\{w(X),w(\Bo),\aleph_{0})\}$
and f.p. maps $g:X\to Z$ and $h:Z\to Y$ such that $f ~ h=g$.
\end{lemma}
\begin{proof}
By Proposition \ref{proposition 0.2} we can assume that $f(X)$ is a closed subset of an $\ANR_{\Bo}$-space $M$, which satisfies the condition $w(M)\leq max(w(f(X)), w(\Bo), \aleph_{0})$.

Since for metric spaces the weight coincides with the density we conclude that 
\[
w(f(X)) = s(f(X))\leq s(X)\leq w(X).
\]

Since $Y$ is an $\ANR_{\Bo}$-space there exist an open neighbourhood $Z$ of $f(X$) in $M$ and a f.p. map $h:Z\to Y$, which extends the f.p. inclusion $i:f(X)\to Y$, i.e. $h~j=i$, where $j:f(X)\hookrightarrow Z$ is the fiber inclusion too. Let $g=j~f^{|X}$. Here $f^{|X}$ is the restriction $f:X\to f(X)$ of map $f$. Note that 
\[h~g=h~j~f^{|X}=i~f^{|X}=f.\]

It is readily checked that $w(Z)\leq max\{w(X),w(\Bo),\aleph_0 \}$ and $Z$ is an $\ANR_{\Bo}$-space.
\end{proof}

\begin{proof}[\textbf{Proof of Theorem \ref{theorem 1.2}}]
We say that two f.p. maps $p:X\to P$, $p^{'}:X\to P^{'}$ are equivalent if there is a f.p. homeomorphism $h:P\to P^{'}$ such that $h ~ p=p^{'}$. Let $\kc$ consist of all equivalence classes of f.p. maps $p:X\to P$, where $P$ is an $\ANR_{\Bo}$-space of weight $w(P)\leq max\{w(X),w(\Bo),\aleph_{0}\}$.

Let $\gamma\in \kc$ and let $p_{\gamma}:X\to Y_{\gamma}$ be a f.p. map from the class $\gamma$. We order the set $\kb$ of all finite subsets $\beta=\{\gamma_1, \gamma_2,\cdots,\gamma_n\}$ of the set $\kc$ by inclusion. The set $\kb$ is a directed set. Let $Y_\beta = Y_{\gamma_1}\times  Y_{\gamma_2}\times  \cdots  \times Y_{\gamma_n}$ be the product of the space $Y_{\gamma_i}$ over $\Bo$, $\gamma_i\in \beta=\{\gamma_1, \gamma_2, \cdots,\gamma_n\}$, $i=1,2, \cdots ,n$ in the category $\bTop_{\bBo}$, i.e. the pull-back of the maps $Y_{\gamma_{i}}\to \Bo$,$i=1,2,\cdots,n$. 

For every $\beta\leq \beta^{'}=\{\gamma_1, \gamma_2,\cdots,\gamma_n, \gamma_{n+1}, \cdots, \gamma_m\}$ we define the f.p. map $q_{\beta \beta^{'}}:Y_{\beta^{'}}\to Y_{\beta}$ as the projection
\[
q_{\beta \beta^{'}}(y_{\gamma_1},  y_{\gamma_2},  \cdots,y_{\gamma_n}, y_{\gamma_{n+1}}, \cdots, y_{\gamma_m})= (y_{\gamma_1},  y_{\gamma_2},  \cdots,y_{\gamma_n}).
\]

Next we define the maps $q_{\beta}:X\to Y_{\beta}$, $\beta \in \kb$ as the maps given by the formula
\[q_{\beta} (x)=(q_{\gamma_1}(x), q_{\gamma_2}(x), \cdots,q_{\gamma_n}(x)),~~~~~x\in X.\]

It is readily seen that $Y_{\beta}$ is $\ANR_{\Bo}$-space and $q_{\beta \beta^{'}} ~ q_{\beta^{'} \beta^{''}}=q_{\beta \beta^{''}}$ for every $\beta \leq \beta^{'}\leq \beta^{''}$.

Note that  $\bY=(Y_{\beta}, q_{\beta\beta^{'}}, \kb)$ is an $\ANR_{\Bo}$-system and $\bq=(q_{\beta}):X\to \bY$ is a morphism of $\bpro-{\bTop_{\bBo}}$.

Condition R$_{\Bo}1)$ is an immediate consequence of Lemma \ref{lemma 1.3}.

We replace the inverse system $\bY$ by a larger inverse system. All pairs $\alpha=(\beta,U)$, where $\beta\in \kb$ and $U$ is an open neighbourhood of $q_{\beta}(X)$ in $Y_{\beta}$, form a directed set $\ka$ provided $\alpha=(\beta,U)\leq (\beta^{'},U^{'})=\alpha^{'}$ means that $\beta \leq \beta^{'}$ and $q_{\beta\beta^\prime}(U^\prime )\subseteq U$.

For every $\alpha=(\beta,U)\in \ka$ and $\alpha\leq \alpha^\prime$ we put $X_{\alpha}=U$, $p_{\alpha}=q_{\beta}:X\to U$ and $p_{\alpha\alpha^{'}}=q_{\beta\beta^{'}|U^{'}}:U^{'}\to U$. We obtain an $\ANR_{\Bo}$ inverse system $\bX=(X_{\alpha}, p_{\alpha\alpha^{'}},\ka)$. It is readily checked that the morphism $\bp=(p_{\alpha}):X\to \bX$ of category $\bpro-\bTop_{\bBo}$ satisfies conditions R$_{\Bo}1)$ and R$_{\Bo}2)$.
\end{proof}

\begin{Definition}\label{Definition 1.4}
\textit{Let $X$ be a topological space over $\Bo$, $\bX=(X_{\alpha},p_{\alpha\alpha^{'}},\ka)$ an inverse system in $\bTop_{\bBo}$ and $\bp=(p_{\alpha}):X\to \bX$ a morphism of $\bpro-\bTop_{\bBo}$. We call $\bp$ an expansion over $\Bo$ of the space $X$ over $\Bo$ provided it has the following properties:}

\item[E$_{\Bo}1).$] \textit{For every $\ANR_{\Bo}$-space $P$ over $\Bo$ and f.p. map $f:X\to P$ there is an index $\alpha\in \ka$ and a f. p. map $h:X_{\alpha}\to P$ such that $h ~ p_{\alpha}\mathop {\simeq}\limits_{\Bo}f$.}

\item[E$_{\Bo}2).$] \textit{If $f, f^{'}:X_{\alpha}\to P$ are f. p. maps, $P\in \ANR_{\Bo}$ and $f ~ p_{\alpha}\mathop {\simeq}\limits_{\Bo}f^{'} ~ p_{\alpha}$, then there is an index $\alpha^{'}\geq \alpha$ such that $f ~ p_{\alpha\alpha^{'}} \mathop {\simeq}\limits_{\Bo} f^{'} ~ p_{\alpha\alpha^{'}}$.}
\end{Definition}

\begin{Definition}\label{Definition 1.5a}
\textit{A morphism $\bp:X\to (X_{\alpha},p_{\alpha\alpha^{'}},\ka)$ is called a strong expansion over $\Bo$ provided it satisfies condition E$_{\Bo}1)$ and the following condition:}

SE$_{\Bo}2).$ \textit{Let $P$ be an $\ANR_{\Bo}$-space, let $f_0,f_1:X_{\alpha}\to P$, $\alpha\in \ka$ be f.p. maps and let $F:X\times I\to P$ be a f.p. homotopy such that}
\[S(x,0)=f_0 p_\alpha (x),~~~~~x\in X,\]
\[S(x,1)=f_1 p_\alpha (x),~~~~~x\in X.\]

\textit{Then there exists a $\alpha^{'}\geq \alpha$ and a f.p. homotopy $H:X_{\alpha^{'}}\times I\to P$, such that 
}\[H(x,0)=f_0p_{\alpha\alpha^{'}}(z),~~~~~z\in X_{\alpha^{'}},\]
\[H(x,1)=f_1p_{\alpha\alpha^{'}}(z),~~~~~z\in X_{\alpha^{'}},\]
\[H(p_{\alpha^{'}}\times 1_{I})\mathop {\simeq}\limits_{\Bo} S(\rel (X\times \partial I)).\]
\end{Definition}

It is clear that, every strong expansion over $\Bo$ is an expansion over $\Bo$.

If all $X_\alpha \in \ANR_{\Bo}$, then $\bp$ is called an $\ANR_{\Bo}$-expansion and strong $\ANR_{\Bo}$-expansion, respectively.

The main result of section 1 is the following theorem.

\begin{theorem}\label{theorem 1.5}
Let $X$ be a topological space over $\Bo$. Then every resolution  $\bp:X \to \bX$ over $\Bo$ induces a strong $\ANR_{\Bo}$-expansion.
\end{theorem}
\begin{corollary}\label{corollary 1.6}
Every $\ANR_{\Bo}$-resolution over $\Bo$ induces $\ANR_{\Bo}$-expansion.
\end{corollary}
\begin{corollary}\label{corollary 1.7}
Every space $X$ over $\Bo$ admits a cofinite strong $\ANR_{\Bo}$-expansion.
\end{corollary}

In the proof of Theorem \ref{theorem 1.5} we need the following lemma.
\begin{lemma}\label{lemma 1.8}
Let $X$ be a topological space over metrizamble space $\Bo$, let $P,P^{'}$ be $\ANR_{\Bo}$-spaces, let $f:X\to P^{'}$, $h_0,h_1:P^{'}\to P$ be f.p. maps and let $S:X\times I\to P$ be a f.p. homotopy such that
\[S(x,0)=h_0f(x),~~~~~~x\in X,\]
\[S(x,1)=h_1f(x),~~~~~~x\in X.\]
Then there exists an $\ANR_{\Bo}$-space $P^{''}$, f.p. maps $f^{'}:X\to P^{''}$, $h:P^{''}\to P^{'}$ and a f.p. homotopy $K:P^{''}\times I\to P$ such that 
\begin{align*}
hf^{'}&=f,\\
K(z,0)&=h_0h(z),~~~~~z\in P^{''}\\
K(z,1)&=h_1h(z),~~~~~z\in P^{''}\\
K(f^{'}&\times 1_{I})=S.
\end{align*}
\end{lemma}
\begin{proof}
Let $S:X\times I\to P$ be a map such that $S(x,0) = (h_0 ~ f)(x)$, $S(x,1)=(h_1 ~ f)(x)$ and $\pi_P ~ S =\pi_{X\times I}$.  Consider the subspace $C_{\Bo}(I,P)$ of the space $C(I,P)$. Let $\pi_{C_{\Bo}(I,P)}:C_{\Bo}(I,P)\to \Bo$ be the map given by
$\pi_{C_{\Bo}(I,P)}(\varphi)=\pi_{P}(\varphi(t)).$

Consequently, $C_{\Bo}(I,P)$ is a space over $\Bo$. The f.p. map $S:X\times I\to P$ defines the map $s:X\to C_{\Bo}(I,P)$ such that $(s(x))(t) = S(x,t)$, $x\in X$, $t\in I$. The image of the point $x\in X$, $s(x)\in C_{\Bo}(I,P)$, because $\pi_{P} ~ s(x):I\to \Bo$ is a constant map. Indeed,
\[(\pi_P ~ s(x))(t)=\pi_P(s(x))(t)=\pi_P(S(x,t))=\pi_{X\times I}(x,t)=\pi_X(x)\]
for every $t\in I$.

For each $x\in X$ we have 
\[(\pi_{C_{\Bo}(I,P)} ~ s)(x)=(\pi_{C_{\Bo}(I,P)}(s(x))=\pi_{P}(s(x))(t)=\] 
\[=\pi_{P}(S(x,t))=\pi_{X\times I}(x,t)=\pi_X(x).\] 

Thus, $\pi_{C_{\Bo}(I,P)} ~ s=\pi_X$. Hence, $s:X\to C_{\Bo}(I,P)$ is a f.p. map. For all $x\in X$ we have $$(s(x))(0) = S(x,0) = (h_0 ~ f)(x)$$ and $$(s(x))(1)=S(x,1) = (h_1 ~ f)(x).$$

Let $P^{'}\times_{\Bo}C_{\Bo}(I,P) =\{(y,\varphi)| \pi_{P^{'}}(y)=\pi_{C_{\Bo}(I,P)}(\varphi)\}$. The map $f^{'}:X\to P^{'}\times_{\Bo}C_{\Bo}(I,P)$, given by $f^{'}(x) =(f(x),s(x))$, is a f.p. map. Let $\pi_{P^{'}\times_{\Bo}C_{\Bo}(I,P)}:P^{'}\times_{\Bo}C_{\Bo}(I,P)\to \Bo$ be a map defined by
\[\pi_{P^{'}\times_{\Bo}C_{\Bo}(I,P)}(y,\varphi)=\pi_{P^{'}}(y)=\pi_{C_{\Bo}(I,P)}(y).\]
Then we have
\[\pi_{P^{'}\times_{\Bo}C_{\Bo}(I,P)} ~ f^{'}=\pi_{P^{'}\times_{\Bo}C_{\Bo}(I,P)}(f(x),s(x))=\pi_{P^{'}}(f(x))=\pi_{X}(x).\]

Thus, $\pi_X=\pi_{P^{'}\times_{\Bo}C_{\Bo}(I,P)} ~ f^{'}$.

It is clear that the first projection $h:P^{'}\times_{\Bo}C_{\Bo}(I,P)\to P^{'}$ is a f.p. map and $h ~ f^{'}=f$.

We define the subset $P^{''}$ of $P^{'}\times_{\Bo}C_{\Bo}(I,P)$ be the following way:
\[P^{''}=\{(y,\varphi)\in P^{'}\times_{\Bo}C_{\Bo}(I,P)|\varphi(0)=h_0(y), h_1(y)=\varphi(1)\}.\]
Let $K:P^{'}\times_{\Bo} C_{\Bo}(I,P)\times I\to P$ be a map given by formula
\[K((y,\varphi),t)=\varphi(t), y\in P^{'}, \varphi \in C_{\Bo}(I,P), t\in P.\]
The restriction of $K$ on $P^{''}\times I$ again denote by $K:P^{''}\times I\to P.$ This map is a f.p. homotopy between $h_0 ~ h_{|P^{''}}$ and $h_1 ~ h_{|P^{''}}$.

Indeed, for every $(y,\varphi)\in P^{''}$ and $t\in I$ we have
\[K((y,\varphi),0)=\varphi(0)=h_0(y)=h_0 ~ h(y,\varphi),\]
\[K((y,\varphi),1)=\varphi(1)=h_1(y)=h_1 ~ h(y,\varphi),\]
\[\pi_{P^{'}\times_{\Bo} C_{\Bo}(I,P)\times I}((y,\varphi),t)=\pi_{P^{'}\times_{\Bo} C_{\Bo}(I,P)}(y,\varphi)=\] \[=\pi_{P^{'}}(y)=\pi_{C_{\Bo}(I,P)}(\varphi)=\pi_{P}(\varphi(t))=\pi_{P}(K(y,\varphi),t).\]
Note that for each $x\in X$ and $t\in I$ 
\[K(f^{'}\times 1_{I})(x,t)=K((f(x),s(x)),t)=(s(x))(t)=(S(x,t)).\]
Hence, $K(f^{'}\times 1_{I})=S.$

We shall prove that $P^{''}\in \ANR_{\Bo}$. Now suppose that $A$ is a closed subspace of a space $Z$ over $\Bo$ and $l:A\to P^{''}$ is a map such that $\pi_A=\pi_{Z|A}=\pi_{P} ~ l.$

Denote by $L:A\times I\to P$ the map defined by
\[L(a,t)=(h^{'} ~ l(a))(t), (a,t)\in  A\times I,\]
where $h^{'}$ is the second projection $P^{'}\times_{\Bo} C_{\Bo}(I,P)\to C_{\Bo}(I,P)$. It is clear that $L$ is a f.p. map. Indeed,
\[(\pi_{P} ~ L)(a,t)=\pi_{P}(L(a,t))=\pi_{P}((h^{'} ~ l(a))(t)=\]
\[=\pi_{C_{\Bo}(I,P)}(h^{'}(l(a)))=\pi_{A}(a)=\pi_{A\times I}(a,t).\]

The map $L$ is a f.p. homotopy between $h_0 ~ h ~ l$ and $h_1 ~ h ~ l$. Indeed, 
\[L(a,0)=(h^{'} ~ l(a))(0)=h_0 ~ h ~ l(a),~~~~a\in A\]
and
\[L(a,1)=(h^{'} ~ l(a))(1)=h_1 ~ h ~ l(a),~~~~a\in A.\]

Observe that, since $P^{'}\in \ANR_{\Bo}$ and $h ~ l:A\to P^{'}$ is a f.p. map, there is a neighbourhood $U$ of $A$ in $Z$ and there exists a f.p. map $\tilde{l}^{'}:U\to P^{'}$ such that $\tilde{l}^{'}_{|_A}=h ~ l$.

By Preposition \ref{proposition 0.4} there exists a neighbourhood $V$ of $A$ in $U$ and a f.p. homotopy $\tilde{L}:V\times I\to P$ between $h_0 ~ \tilde{l}^{'}_{|V}$ and $h_1 ~ \tilde{l}^{'}_{|V}$. Also note that $\tilde{L}(a,t)=L(a,t)$ for each $a\in A$ and $t\in I$. Let $\tilde{l}^{''}$ be a f.p. map $\tilde{l}^{''}:V\to C_{\Bo}(I,P)$, given by $(\tilde{l}^{''}(z))(t)=\tilde{L}(z,t)$, $z\in V, t\in I$. For every $a\in A$ we have $$(\tilde{l}^{''}(a))(t)=\tilde{L}(a,t)=L(a,t)=(h^{'} ~ l(a))(t).$$ 

Consequently, $\tilde{l}^{''}_{|{A}}=h^{'} ~ l$. Now define the f.p. map $\tilde{l}:V\to P^{'}\times_{\Bo} C_{\Bo}(I,P)$ by the formula $$\tilde{l}(z)=(\tilde{l}^{'},\tilde{l}^{''}),~~z\in V.$$ 

For each $z\in V$ we have
\[(\tilde{l}^{''}(z))(0)=\tilde{L}(z,0)=h_0 ~ \tilde{l}^{'}(z),\]
\[(\tilde{l}^{''}(z))(1)=\tilde{L}(z,1)=h_1 ~ \tilde{l}^{'}(z).\]
Consequently, $\tilde{l}:V\to P^{''}$ is an extension of the f.p. map $l:A\to P^{''}$. This fact completes the proof of lemma \ref{lemma 1.8}.
\end{proof}
\begin{lemma}\label{lemma 1.9}
Let $\bp:X\to \bX$ be a resolution over $\Bo$ and let $\alpha,P,f_0,f_1$ and $F$ be as in $\SE_{\Bo}2)$. Then for every open covering $\mathscr{U}$ of $P$, there exist a $\alpha^{'}\geq \alpha$ and a f.p. homotopy $H:X_{\alpha^{'}}\times I\to P$ such that
\begin{align*}
H(y,0)&=f_0 ~ p_{\alpha\alpha^{'}}(y),~~~~~~~~~y\in X_{\alpha^{'}}\\
H(y,1)&=f_1 ~ p_{\alpha\alpha^{'}}(y),~~~~~~~~~y\in X_{\alpha^{'}}\\
(S,H(1&\times p_{\alpha^{'}}))\leq \mathscr{U}
\end{align*}
\end{lemma}
\begin{proof}
Let $\mathscr{U}$ be an open covering of $P$. There exists an open star-refinement $\mathscr{U}^{'}$ of $\mathscr{U}$. Now we choose an open covering $\mathscr{V}$ of $P$ such that the assertions of Proposition \ref{proposition 0.3} hold for $\mathscr{U}^{'}$. We can assume that $\mathscr{V}$ is a star-refinement of $\mathscr{U}^{'}$. We choose $\mathscr{V}^{'}$ so that $\mathscr{V}^{'}$ is a star-refinement of $\mathscr{V}$ and R$_{\Bo}2)$ holds for $P$, $\mathscr{V}$ and $\mathscr{V}^{'}$.

Let $P^{'}=P\times_{\Bo}P$. By $g_0,g_1:P^{'}\to P$ denote the two projections. Let $f:X\to P^{'}$ be the diagonal product of f.p. maps $f_0 ~ p_{\alpha}:X\to P$ and $f_1 ~ p_{\alpha}:X\to P$. It is clear that $g_0 ~ f=f_0 ~ p_{\alpha}$, $g_1 ~ f=f_1 ~ p_{\alpha}$, $F_0=g_0 ~ f$ and $F_1=g_1 ~ f$.

By the lemma \ref{lemma 1.8} there exists an $\ANR_{\Bo}$-space $P^{''}$, f.p. maps $f^{'}:X\to P^{''}$, $g:P^{''}\to P^{'}$ and a f.p. homotopy $G:P^{''}\times I\to P$ such that 
\[g ~ f^{'}=f,\]
\[G_0=g_0 ~ g, G_1=g_1 ~ g,\]
\[G(f^{'}\times 1)=F.\]

We choose for the open covering $G^{-1}(\mathscr{V}^{'})$ of $P^{''}\times I$ a refinement, which is a stacked covering $\kv$ of  $P^{''}\times I$, given by a locally finite open covering $\mathscr{W}$ of $P^{''}$ and by finite open coverings $\mathscr{J}_\mathscr{W}, W\in \kw$ of $I$.

By condition R$_{\Bo}1)$ there exists a $\alpha^{''}\geq \alpha$ and f.p. mapping $h:X_{\alpha^{''}}\to P^{''}$ such that \[(f^{'},h  ~  p_{\alpha^{''}})\leq \kw\]

It is clear that for any $W\in \mathscr{W}$, $W\times 0\subseteq W\times J$, where $J\in \mathscr{J}_{\mathscr{W}}$ and $W\times J\subset G^{-1}(V^{'})$ for some $V^{'}\in \mathscr{V}^{'}$.

Note that
\[g_0g(W)=G_0(W)=G(W\times 0)\subseteq G(W\times J)\subseteq V^{'}.\]

Hence, $g_0g(\kw)$ refines $\kv^{'}$ and $(g_0~g~f^{'},g_0~g~h~p_{\alpha^{''}})\leq \kv^{'}.$

From the equalities 
\[g_0 ~ g f^{'}=g_0 f=f_0 p_{\lambda}=f_0 p_{\alpha \alpha^{'}}p_{\alpha^{''}}\]
it follows that 
\[(g_0 g h p_{\alpha^{''}}, f_0 ~ p_{\alpha \alpha^{''}} ~ p_{\alpha^{''}})\leq \mathscr{V}^{'}.\]

We also can claim that
\[(g_1 g h p_{\alpha^{''}}, f_1 ~ p_{\alpha \alpha^{''}} ~ p_{\alpha^{''}})\leq \mathscr{V}^{'}.\]

By condition R$_{\Bo}2)$ there is a $\alpha^{'}\geq \alpha^{''}$ such that 
\[(g_0 g h p_{\alpha^{''}\alpha^{'}}, f_0 ~ p_{\alpha \alpha^{'}})\leq \mathscr{V}\]
and 
\[(g_1 g h p_{\alpha^{''}\alpha^{'}}, f_1 ~ p_{\alpha \alpha^{'}})\leq \mathscr{V}.\]

Besides, there exist $\mathscr{U}^{'}$-f.p. homotopies $K,L:X_{\alpha^{'}}\times I\to P$ such that $K_0=f_0 p_{\alpha\alpha^{'}}, K_1=g_0 g h p_{\alpha^{''}\alpha^{'}}$, $L_0=f_1p_{\alpha \alpha^{'}}$ and $L_1=g_1ghp_{\alpha^{''}\alpha^{'}}.$

Note that for any $t\in I$ the pairs $(f^{'}(x),t)$ and $(hp_{\alpha^{''}}(x),t)$ belong to some elements of $\kv$ and consequently to $G^{-1}(V^{'})$ for some $V^{'}\in \mathscr{V}^{'}$. Thus $G(f^{'}\times 1_{I})$ and $G(h p_{\alpha^{''}}\times 1_{I})$ are $\mathscr{V}^{'}$-near. Hence,
\[(G(f^{'}\times 1_{I}), G(hp_{\alpha^{''}}\times 1_{I}))\leq \mathscr{V}.\]

Now we define f.p. homotopy $H:X_{\alpha^{'}}\times I\to P$ by formulas
\begin{equation*}
H(y,t)=\begin{cases}
K(y,\frac{t}{\varphi(z)}),& 0\leq t \leq \varphi(z), \\
G(z,\frac{t-\varphi(z)}{1-2\varphi(z)}),& \varphi\leq t\leq 1-\varphi(z),\\
L(y,\frac{1-t}{\varphi(z)}),& 1-\varphi(z)\leq t\leq 1,
\end{cases}
\end{equation*}
where $z=hp_{\alpha^{''}\alpha^{'}}(y)$ and $\varphi:P^{''}\to I$ is a continuous map defined in [M$_1$].

As in [M$_1$] we can prove that for every $(x,t)\in X\times I$, there is a $U\in \ku$ such that \[F(x,t), H(p_{\alpha}(x),t)\in U.\]
\end{proof}



\begin{proof}[\textbf{Proof of Theorem \ref{theorem 1.5}}]. First prove the following condition.

E$_{\Bo}1)$. Let $\mathscr{U}$ be a open covering of $P$. Consider open covering $\mathscr{V}$ as in Proposition \ref{proposition 0.3}. By R$_{\Bo}1)$ there exist an index $\alpha\in \ka$ and a f.p. mapping $h:X_{\alpha}\to P$ which satisfies condition $(hp_{\alpha},f)\leq \mathscr{V}$. Thus, by the choice of $\mathscr{V}$, $f \mathop {\simeq}\limits_{\Bo}h ~ p_{\alpha}$.

S$_{\Bo}2)$. Let $\mathscr{U}$ be a open covering $\mathscr{U}$ of. Consider a covering $\mathscr{V}$ as in Proposition \ref{proposition 0.3}. By Lemma \ref{lemma 1.8} there exist a $\alpha^{'}\geq \alpha$ and f.p. homotopy $H:X_{\alpha^{'}}\times I\to P$ which satisfies 
\[H(z,0)=f_0p_{\alpha \alpha^{'}}(z),~~~~z\in X_{\lambda^{'}},\]
\[H(z,1)=f_1p_{\alpha \alpha^{'}}(z),~~~~z\in X_{\lambda^{'}},\]
\[(S,H(1\times p_{\alpha^{'}}))\leq \mathscr{V} .\]

Consider the spaces $Z=X\times I$ and $A=X\times \partial I$ over $\Bo$ and f.p. mappings $h_0=F$ and $h_1=H(p_{\alpha^{'}}\times 1)$. 

Note that $h_{0_|A}=h_{1_|A}$. Indeed, for each $x\in X$ 
\[h_0(x,0)=F(x,0)=f_0p_{\alpha}(x)=f_0p_{\alpha \alpha^{'}}p_{\alpha^{'}}(x)=H(p_{\alpha^{'}} (x),0)=h_1(x,0).\]

Analogously, for each $x\in X$ we have 
\[h_0(x,1)=F(x,1)=f_1p_{\alpha}(x)=f_1p_{\alpha \alpha^{'}}p_{\alpha^{'}}(x)=H(p_{\alpha^{'}} (x),1)=h_1(x,0).\]

Consequently, $(h_0,h_1)\leq \mathscr{V}$. By Preposition \ref{proposition 0.3} there exists a f.p. homotopy $\rel (X\times \partial I)$, which connects $F$ and $H(p_{\alpha^{'}}\times 1_{I})$.
\end{proof}
\section{On Fiber Strong Shape Category }

Let $\Delta^{n}$ be the standard $n$-simplex, i.e. the set of all points $t=\{ t=(t_0,t_1, \cdots,t_n)\in R^{n+1}\}$, where $t_0\geq 0, \cdots, t_n\geq 0$ and $t_0+ \cdots +t_n=1$.

For $n>0$ and $0\leq j\leq n$ there exist $\partial^{n}_{j}:\Delta^{n-1}\to \Delta^{n}$ $j$-th face operators and for $n\geq 0$ and $0\leq j\leq n$ there exist $\sigma^{n}_{j}:\Delta^{n+1}\to \Delta^{n}$ $j$-th degeneracy operators given by formulas
\[\partial^{n}_{j}(t_0,\cdots,t_{n-1})=(t_0,\cdots,t_{j-1},0,t_j,  \cdots, t_{n-1}), \]
\[\sigma^{n}_{j}(t_0, \cdots,t_{n+1})=(t_0, \cdots,t_{j-1},t_j+t_{j+1},t_{j+2},  \cdots, t_{n+1}). \]

Let $\kb$ be a directed set. By $\kb^{n}$ denote the set of all sequences $\hm{\beta}=(\beta_0,\cdots,\beta_n)$, $\beta_0\leq \cdots\leq\beta_n$ of elements of $\kb$.

For $n>0$ and $0\leq j\leq n$ we consider the $j$-th face operator $d^{n}_{j}:\kb^{n}\to \kb^{n-1}$ given by formula 
\[d^{n}_{j}(\beta_0, \cdots,\beta_n)=(\beta_0, \cdots,\beta_{j-1},\beta_{j+1}, \cdots,\beta_n)\]
and for $n\geq 0$ and $0\leq j\leq n$ by $s^{n}_{j}$ we denote $j$-th degeneracy operator $s^{n}_{j}:\kb^{n}\to \kb^{n+1}$ given by formula
\[s^{n}_{j}(\beta_0, \cdots,\beta_n)=(\beta_0, \cdots,\beta_{j},\beta_{j}, \cdots,\beta_n).\]

For simplicity the images $d^{n}_{j}(\hm{\beta})$ and $s^{n}_{j}(\hm{\beta})$ we denote by $\hm{\beta}_{j}$ and $\hm{\beta}^{j}$, respectively.

Let $\bX=(X_{\alpha}, p_{\alpha \alpha^{'}}, \ka)$ and $\bY=(Y_{\beta}, p_{\beta \beta^{'}}, \kb)$ be the objects of category $\bpro-\bTop_{\bBo}$.

A coherent map $f:{\bf X}\to \bY$ over $\rm B_0 $ or fiber preserving (f.p) coherent map consists of function $\varphi: \kb^{n}\to \ka$ and fiber preserving maps $f_{\hm{\beta}}:X_{\varphi(\hm{\beta})}\times \Delta^{n}\to Y_{\beta_0}$, $\hm{\beta}=(\beta_0, \cdots,\beta_n)\in \kb^{n}$, $n\geq 0$ having the following properties:

\begin{itemize}
\item [i).]The function $\varphi$, which assigns to every $n\geq 0$ and $\hm{\beta}=(\beta_0, \cdots,\beta_n)\in \kb^{n}$ an element $\varphi(\hm{\beta})=\varphi(\beta_0, \cdots,\beta_n)\in \ka$, satisfies condition:
\[
\varphi(\hm{\beta})\geq \varphi(\hm{\beta}_{j}),~~~~~~~0\leq j\leq n, n>0.
\]

\item[ii).] For every $n\geq 0$ and every $\bm{\beta}=(\beta_0, \cdots,\beta_n)\in \kb^{n}$ the fiber preserving maps $f_{\hm{\beta}}:(X_{\varphi(\hm{\beta})}\times \Delta^{n},\pi_{X_{\varphi(\hm{\beta})}\times \Delta^{n}})\to (Y_{\beta_0},\pi_{Y_{\beta_0}})$ satisfies condition:
\[
f_{\hm{\beta}}(x,\partial^{n}_{j}t)=\begin{cases}
q_{\beta_0\beta_1}f_{\hm{\beta_0}}(p_{\varphi(\hm{\beta_0})}p_{\varphi(\hm{\beta})}(x),t),~~~~~j=0\\
f_{\hm{\beta_j}}(p_{\varphi(\hm{\beta_j})}p_{\varphi(\hm{\beta})}(x),t),~~~~~0\leq j\leq n,
\end{cases}
\]
\end{itemize}
where $x\in X_{\varphi(\hm{\beta})}$, $t\in \Delta^{n-1}$, $n\geq 0$, $X_{\varphi(\hm{\beta})}\times \Delta^{n}$ is the space over $\rm B_0 $ with projection $\pi_{X_{\varphi(\hm{\beta})}\times \Delta^{n}}:X_{\varphi(\hm{\beta})}\times \Delta^{n}\to {\rm B_0} $ given by formula
\[\pi_{X_{\varphi(\hm{\beta})}\times \Delta^{n}}(x,t)=\pi_{X_{\varphi(\hm{\beta})}}(x),~~~~~~x\in X_{\varphi(\hm{\beta})}, t\in \Delta ^{n}\]
and 
\[f_{\hm{\beta}}(p_{\varphi(\hm{\beta})\varphi(\hm{\beta_j})}(x),\sigma^{n}_{j}(t))=f_{\hm{\beta^{j}}}(x,t), 0\leq j\leq n, x\in X_{\varphi(\hm{\beta^{j}})}, t\in \Delta ^{n+1}, n\geq 0.\]

The identity coherent map $1_{{\bf X}}:{\bf X}\to {\bf X}$ over $\rm B_0 $ is given by formulas:
\[\varphi(\hm{\alpha})=\alpha_n,\bm{\alpha}=(\alpha_0, \cdots,\alpha_n)\in \ka^{n},\]
\[1_{\hm{\alpha}}(x,t)=p_{\alpha_0\alpha_n}(x),x\in X_{\alpha_{n}},t\in \Delta^{n},n\geq 0.\]

A coherent homotopy over $\rm B_0 $ or fiber preserwing (f.p.) homotopy $F:{\bf X}\times I\to \bY$ connecting f.p. coherent maps $f,f^{'}:{\bf X}\to \bY$, is a f.p. coherent map of ${\bf X}\times I=((X_{\alpha}\times I,\pi_{X_{\alpha}\times I}), p_{\alpha \alpha^{'}}\times 1_{I}, \ka)$ to $\bY$, given by a function $\Phi$ and by f.p. maps $F_{\hm{\beta}}:(X_{\varphi(\hm{\beta})}\times I\times \Delta_{n},\pi_{X_{\varphi(\hm{\beta^{j}})}\times I\times \Delta_{n}})\to (Y_{\beta_0},\pi_{Y_{\beta_0}})$, witch have i) and ii) properties and satisfy the conditions
\[\Phi(\hm{\beta})\geq \varphi(\hm{\beta}), \varphi^{'}(\hm{\beta}),\]
\[F_{\hm{\beta}}(x,0,t)=f_{\hm{\beta}}(p_{\varphi(\hm{\beta})\Phi(\hm{\beta})}(x),t),\]
\[F_{\hm{\beta}}(x,1,t)=f^{'}_{\hm{\beta}}(p_{\varphi^{'}(\hm{\beta})\Phi(\hm{\beta})}(x),t),\] 
where $x\in X_{\varphi(\hm{\beta})}, t\in \Delta^{n}, n\geq 0$.

As in \cite{78} we can prove

\begin{proposition}\label{2.1}
\textit{The f.p. coherent homotopy relation of f.p. coherent maps is an equivalence relation.}
\end{proposition}
A f.p. coherent map $f:{\bf X}\to \bY$ is called a special  f.p. coherent map or a special coherent map over $\rm B_0 $ if $\varphi(\hm{\beta})=\varphi(\beta_{n})$ for each $\hm{\beta} \in \kb^{n}$ and $\varphi_{|\kb}:\kb\to \ka$ is an increasing function.

The composition $h= g~f$ of special f.p. coherent maps over $\rm B_0 $ is defined as in \cite{78}.

A special f.p. coherent homotopy connecting two special f.p. coherent maps $f,f^{'}:{\bf X}\to \bY$ is a f.p. coherent homotopy $F:{\bf X}\times I\to \bY$ between $f$ and $f^{'}$ and at the same time it is a special f.p. coherent map.

Note that if the index set $\kb$ of $\bY$ is cofinite, then special f.p. coherent homotopy relation of special f.p. coherent maps is an equivalence relation. 

The proofs of the following proposition pass as in \cite{78}.

\begin{proposition}\label{2.2}
\textit{Let $f,f^{'}:{\bf X}\to \bY$, $g,g^{'}:\bY\to \bZ=((Z_{\gamma},\pi_{Z_{\gamma}}),r_{\gamma \gamma^{'}},\kc)$ be special f.p. coherent maps and let $F,G$ be special f.p. coherent homotopies connecting $f$ with $f^{'}$ and $g$ with $g^{'}$, respectively. If the index set $\kc$ is cofinite, then there is a special f.p. coherent homotopy connecting $g ~ f$ and $g^{'} ~ f^{'}$.}\qed
\end{proposition}
\begin{proposition}\label{2.3}
\textit{If $f:{\bf X}\to \bY$, $g:\bY\to \bZ$ and $h:\bZ\to \bW$ are special f.p. coherent maps of inverse systems of ${\bf Top}_{{\bf B_0}}$ over cofinite index sets, then there is a special f.p. coherent homotopy connecting $h(gf)$ with $(hg)f$.}\qed
\end{proposition}

\begin{proposition}\label{2.4}
\textit{If $f:{\bf X}\to \bY$ is a special f.p. coherent map of inverse systems of ${\bf Top}_{{\bf B_0}}$
over cofinite index sets and $1_{{\bf X}}$ and $1_{\bY}$ are the f.p. coherent identity maps, then there exist special f.p. coherent homotopies connecting $f ~ 1_{{\bf X}}$ with $f$ and $1_{\bY} ~ f$ with $f$.}\qed
\end{proposition}

As in \cite{78} we can show that whenever the index set $\kb$ of $\bY$ is cofinite, then every f.p. coherent homotopy class $[f]:{\bf X}\to \bY$ of f.p. coherent maps $f:{\bf X}\to \bY$ contains a unique f.p. coherent homotopy class of special f.p. coherent maps. Consequently, in the cofinite case one can define composition of f.p. coherent homotopy classes by composing their special representatives.

Now define the following category. The f.p. coherent prohomotopy category ${\bf CPH}{\bf Top}_{{\bf B_0}}$ has as objects inverse systems ${\bf X}=((X_{\alpha},\pi_{X_{\alpha}}),p_{\alpha \alpha^{'}},\ka)$ of topological spaces over $\rm B_0 $ and f.p. maps over directed cofinite index sets. The morphisms are f.p. coherent homotopy classes $[f]:{\bf X}\to \bY$ of f.p. coherent maps $f:{\bf X}\to \bY$ of such systems. Composition is defined by composing representatives, which are special f.p. coherent maps. Identity morphism of ${\bf X}$ is the class, containing the coherent map $1_{{\bf X}}:{\bf X}\to {\bf X}$.

Now define the functor $\C:\bpro-{\bf Top}_{{\bf B_0}}\to {\bf CPH}{\bf Top}_{{\bf B_0}}$. Let $(f_{\beta},\varphi):{\bf X}\to \bY$ be a map of inverse systems. We associate with $(f_{\beta},\varphi)$ a f.p. coherent map $f:{\bf X}\to \bY$. For this aim we extend $\varphi:\kb\to \ka$ to a function $\varphi$ defined for all $\hm{\beta}=(\beta_0, \cdots ,\beta_{n})$ in such a way that 
\[\varphi(\hm{\beta})\geq \varphi(\beta_{j}), 0\leq j\leq n.\]

We use the method of induction. Let $n=1$ and $\hm{\beta}=(\beta_0,\beta_1)$. Note that 
\[f_{\beta_0} ~ p_{\varphi(\beta_0)\varphi(\beta_0,\beta_1)}=q_{\beta_0\beta_1}f_{\beta_1} ~p _{\varphi(\beta_1)\varphi(\beta_0,\beta_1)}.\]

Let $f_{\hm{\beta}}:(X_{\varphi(\hm{\beta})}\times \Delta^{n},\pi_{X_{\varphi(\hm{\beta})}\times \Delta^{n}})\to (Y_{\beta_{0}},\pi_{Y_{\beta_{0}}})$ a f.p. mapping defined by
\[f_{\hm{\beta}}(x,t)=f_{\beta_0}p_{\varphi(\beta_0)\varphi(\hm{\beta})}(x),~~~~~x\in X_{\varphi(\hm{\beta})}, t\in \Delta^{n}.\]

Also note that 
\[f_{\hm{\beta}}(x,\partial^{n}_{0}t)=f_{\beta_0}p_{\varphi(\beta_0)\varphi({\hm{\beta}})}(x)=q_{\beta_0\beta_1}f_{\beta_1}p_{\varphi(\beta_1)\varphi({\hm{\beta}})}(x)=q_{\beta_0\beta_1}f_{\hm{\beta_0}}(p_{\varphi(\hm{\beta_0})\varphi(\hm{\beta})}(x),t)\] 
and 
\[f_{\hm{\beta}}(x,\partial^{n}_{j}t)=f_{\beta_0}p_{\varphi(\beta_0)\varphi({\hm{\beta}})}(x)=f_{\hm{\beta_j}}(p_{\varphi(\hm{\beta_j})\varphi(\hm{\beta})}(x),t), ~~~0<j\leq n,\]
\[f_{\hm{\beta}}(p_{\varphi(\hm{\beta})\varphi(\hm{\beta^j})}(x),\sigma^{n}_{j}t)=f_{\beta_0}p_{\varphi(\beta_0)\varphi({\hm{\beta^{j}}})}(x)=f_{\hm{\beta^j}}(x,t), ~~~0\leq j\leq n.\]

Let $\varphi^{'}$ be another extension of $\varphi$. We obtain another f.p. coherent map $f^{'}$. Note that $f$ and $f^{'}$ are f.p. coherently homotopic.

Let $(f_{\beta},\varphi), (f^{'}_{\beta},\varphi^{'}):{\bf X}\to \bY$ are equivalent morphisms. As in \cite{78} we can show that the associated f.p. coherent maps $f$ and $f^{'}$ are connected by some f.p. coherent homotopy $F:{\bf X}\times I\to \bY$.

Thus, to every morphism of $\bff:{\bf X}\to \bY$ of $\bpro-{\bf Top}_{{\bf B_0}}$ we can associate a morphism $[f]=\C(\bff)$ of ${\bf CPH}{\bf Top}_{{\bf B_0}}$. If we restrict $\bpro-{\bf Top}_{{\bf B_0}}$ to inverse systems over cofinite index sets, then we have defined a functor $\C:\bpro-{\bf Top}_{{\bf B_0}}\to {\bf CPH}{\bf Top}_{{\bf B_0}}$.

By definition, 
\[\C(\bff)=[f], ~~~~\bff \in {\rm {Mor}}_{\bpro-{\bf Top}_{{\bf B_0}}}({\bf X},\bY),\]
\[\C({\bf X})={\bf X},~~~~\bX\in ob(\bpro-{\bf Top}_{{\bf B_0}}).\]

$\C(1_{\bY})$ is the f.p. coherent homotopy class of $1_{\bY}$. Let $\bff:{\bf X}\to \bY$ and $\bgg:\bY\to \bZ$ be morphism of $\bpro-{\bf Top}_{{\bf B_0}}$. As in \cite{78} we can prove that $\C(\bgg ~ \bff)=\C(\bgg) ~ \C(\bff).$

Besides, there exists a functor $\E:{\bf CPH}{\bf Top}_{{\bf B_0}}\to \bpro-{\bf HTop}_{{\bf B_0}}$. Assume that for each inverse system ${\bf X}=(X_{\alpha},p_{\alpha \alpha^{'}},\ka)$ in ${\bf Top}_{{\bf B_0}}$, $\E {\bf X}=(X_{\alpha},[p_{\alpha \alpha^{'}}]_{{\rm B_0} },\ka)$.

Let $f:{\bf X}\to \bY$ be a f.p. coherent map given by $f_{\hm{\beta}}$ and $\varphi$. We associate with $f$ the morphism $\bff:{\bf X}\to \bY$ of $\bpro-{\bf HTop}_{{\bf B_0}}$, given by function $\varphi_{|_{\kb}}:\kb\to \ka$ and the fiber homotopy classes over $\rm B_0 $,$[f_{\beta_0}]_{{\rm B_0} }:X_{\varphi(\beta_0)}\to Y_{\beta_0}$.

Note that $\bff$ is a morphism of $\bpro-{\bf HTop}_{{\bf B_0}}$. Indeed, for $\beta_0\leq \beta_1$ and $\alpha=\varphi(\beta_0,\beta_1)$ we have $\alpha\geq \varphi(\beta_0),\varphi(\beta_1)$. Besides, the f.p. map $f_{\beta_0\beta_1}:(X_{\alpha}\times \Delta^{1},\pi_{X_{\alpha}\times \Delta^{1}})\to (Y_{\beta_0},\pi_{Y_{\beta_0}})$ satisfies the conditions
\[f_{\beta_0\beta_1}(x,\partial^{1}_{0}(1))=q_{\beta_0\beta_1}f_{\beta_1}(p_{\varphi(\beta_1)\alpha}(x),1)\]
and
\[f_{\beta_0\beta_1}(x,\partial^{1}_{1}(1))=f_{\beta_0}(p_{\varphi(\beta_0)\alpha}(x),1).\]

Thus, 
\[[f_{\beta_0}]_{{\rm B_0} } ~ [p_{\varphi(\beta_0)\alpha}]_{{\rm B_0} }=[q_{\beta_0\beta_1}]_{{\rm B_0} } ~ [f_{\beta_1}]_{{\rm B_0} } ~ [p_{\varphi(\beta_1)\alpha}]_{{\rm B_0} }.\]

Let $f,f^{'}:{\bf X}\to \bY$ be f.p. coherent homotopic maps. Let $F:{\bf X}\times I\to \bY$ be a f.p. coherent homotopy between $f$ and $f^{'}$, given by $\Phi$ and $F_{\hm{\beta}}$. Note that $\Phi(\beta_0)\geq \varphi(\beta_0),\varphi^{'}(\beta_0)$ and $F_{\beta_0}:X_{\Phi(\beta_0)\times I\times \Delta^{0}}\to Y_{\beta_0}$ is a f.p. map satisfying conditions
\[F_{\beta_0}(x,0,1)=f_{\beta_0}(p_{\varphi(\beta_0)\Phi(\beta_0)}(x),1)\]
and
\[F_{\beta_0}(x,1,1)=f^{'}_{\beta_0}(p_{\varphi^{'}(\beta_0)\Phi(\beta_0)}(x),1).\]

Consequently,
\[[f_{\beta_0}]_{{\rm B_0} } ~ [p_{\varphi(\beta_0)\Phi(\beta_0)}]_{{\rm B_0} }=[f^{'}_{\beta_0}]_{{\rm B_0} } ~ [p_{\varphi^{'}(\beta_0)\Phi(\beta_0)}]_{{\rm B_0} }.\]

Thus, with $f$ and with $f^{'}$ is associated the same morphism of $\bpro-{\bf HTop}_{{\bf B_0}}$. Consequently, it is possible to define a functor $\E:{\bf CPH}{\bf Top}_{{\bf B_0}}\to \bpro-{\bf HTop}_{{\bf B_0}}$.

The composition $\E \circ \C:\bpro-{\bf Top}_{{\bf B_0}}\to \bpro-{\bf HTop}_{{\bf B_0}}$ is the functor induced by the f.p. homotopy functor $H:{\bf Top}_{{\bf B_0}}\to {\bf HTop}_{{\bf B_0}}$.

A f.p. coherent map $f:X\to \bY$ consists of f.p. maps $f_{\hm{\beta}}:(X\times \Delta^{n},\pi_{X\times \Delta^{n}})\to (Y_{\beta_0},\pi_{Y_{\beta_0}})$,$\hm{\beta}=(\beta_0,  \cdots, \beta_n)\in \kb$, $n\geq 0$, satisfying the following conditions: 
for each $x\in X$, $t\in \Delta^{n-1},n>0$

\begin{equation*}
f_{\hm{\beta}}(x,\partial^{n}_{j}t)=\begin{cases}
q_{\beta_0\beta_1}f_{\hm{\beta_0}}(x,t),~~~j=0,\\
f_{\hm{\beta_{j}}}(x,t),~~~~~~~~~0<j\leq n
\end{cases}
\end{equation*}
and for each $x\in X$, $t\in \Delta^{n+1},n\geq0$
\[f_{\hm{\beta}}(x,\sigma^{n}_{j}t)=
f_{\hm{\beta^{j}}}(x,t),~~~0\leq j\leq n.\]

Note that a f.p. coherent map $f:X\to \bY$ is always a special f.p. coherent map.

A f.p. coherent homotopy $F:X\times I\to \bY$, connecting $f$ and $f^{'}$, is a f.p. coherent map given by $F_{\hm{\beta}}$ and satisfying the conditions: 
for each $x\in X, t\in \Delta^{n}$
\[F_{\hm{\beta}}(x,0,t)=f_{{\hm{\beta}}}(x,t)\]
and 
\[F_{\hm{\beta}}(x,1,t)=f^{'}_{{\hm{\beta}}}(x,t).\]

Let $\bp=(p_{\alpha}):X\to {\bf X}$ be a morphism of $\bpro-{\bf Top}_{{\bf B_0}}$. It is clear that with $\bp$ is associated a unique f.p. coherent map $p:X\to {\bf X}$ given by formula
\[p_{\hm{\alpha}}(x,t)=p_{\alpha_0}(x),\]
where $\hm{\alpha}=(\alpha_0, \cdots,\alpha_n)\in \ka^{n}, x\in X, t\in \Delta^{n}$.

The objects of category ${\bf SSH}_{{\bf B_0}}$ are all topological spaces over $\rm B_0 $. The morphisms of category ${\bf SSH}_{{\bf B_0}}$ are defined by  the following way.

Let $\bp:X\to {\bf X}$ and $\bq:Y\to \bY$ be an ${\rm {ANR}}_{{\rm B_0} }$-resolutions of $X$ and $Y$, respectively. Let $[f]:{\bf X}\to \bY$ be a some morphism of category ${\bf CPH}{\bf Top}_{{\bf B_0}}$. Let $\bp^{'}:X\to {\bf X}^{'}$, $\bq^{'}:Y\to \bY^{'}$,$[f^{'}]:{\bf X}^{'}\to \bY^{'}$ be another triple of fiber resolutions of spaces $X$ and $Y$ over $\rm B_0 $ and morphism of category ${\bf CPH} {\bf Top}_{{\bf B_0}}$.

Now define the following equivalence relation. We say the triples ($\bp, \bq, [f]$) and ($\bp^{'}, \bq^{'}, [f^{'}]$) are equivalent if 
\[[f^{'}] ~ [i]=[j] ~ [f],\]
where $[i]:\bX \to {\bf X}^{'}$ and $[j]:\bY \to \bY^{'}$ are isomorphisms of category ${\bf CPH}{\bf Top}_{{\bf B_0}}$.

The fiber strong shape morphisms $\F:(X,\pi_{X})\to (Y,\pi_{Y})$ are the equivalence classes of triples ($\bp, \bq, [f]$) with respect to the above defined relation $\sim$. 

Let $F:(X,\pi_{X})\to (Y,\pi_{Y})$ and $G:(Y,\pi_{Y})\to (Z,\pi_{Z})$ be the fiber strong shape morphisms, defined by triples ($\bp, \bq, [f]$) and ($\bp^{'}, \bq^{'}, [g]$), where $\bp^{'}:(Y,\pi_Y)\to \bY^{'}$, $\bq^{'}:(Z,\pi_{Z})\to \bZ$ and $[g]:\bY^{'}\to \bZ$.

As we know there exists an unique morphism $[h]:\bY\to \bY^{'}$ of category ${\bf CPH}{\bf Top}_{{\bf B_0}}$ such that $[h] ~ [q]=[q^{'}]$. Note that 
\[[j][q]=[q^{'}]=[h] ~ [q].\]
Hence, $[j]=[h]$. Besides, $[g] ~ [j]=[g] ~ [h] ~[1_{\bZ}]$.

Thus, we can assume that the morphisms $F$ and $G$ are given by triples ($\bp, \bq, [f]$) and ($\bq, \br, [g]$).

Consequently, we can define the composition $G ~ F:X\to Z$ as the morphism given by triple ($\bp, \br, [g] ~ [f]$).

In the role an identity morphism $\ki:X\to X$ we can take the morphism defined by triple ($\bp, \bp, [1_{X}]$).

The obtained category ${\bf SSH}_{{\bf B_0}}$ call the fiber strong shape category.

Let $X\in ob({\bf SSH}_{{\bf B_0}})$. By symbol $\pssh_{{\rm B_0} }(X)$ denote the equivalence class of topological space $(X,\pi_{X})$ and call the fiber strong shape of $(X,\pi_{X})$.

For each f.p. map $\varphi:(X,\pi_{X})\to (Y,\pi_{Y})$ choose ${\rm {ANR}}_{{\rm B_0} }$-resolutions $\bp:(X,\pi_{X})\to {\bf X}$ and $\bq:(Y,\pi_{Y})\to \bY$. There exists a unique morphism $[f]:{\bf X}\to \bY$ of category ${\bf CPH}{\bf Top}_{{\bf B_0}}$ such that $[q] ~ [\varphi]=[f] ~ [p]$.

We can define a functor $\Ss^{'}_{{\rm B_0} }:{\bf Top}_{{\bf B_0}}\to {\bf SSH}_{{\bf B_0}}$. By definition,
\[\Ss^{'}(X)=X,~~~~~X\in ob({\bf Top}_{{\bf B_0}})\]
and 
\[\Ss^{'}(\varphi)=\Phi,~~~~~\varphi\in {\rm {Mor}}_{{\bf Top}_{{\bf B_0}}}(X,Y).\]

Here $\Phi$ is a fiber strong shape morphism defined by triple ($\bp, \bq, [f]$).

As in \cite{78} we can prove that functor $\Ss^{'}_{{\rm B_0} }$ induces a functor $\Ss_{{\rm B_0} }:{\bf HTop}_{{\bf B_0}}\to {\bf SSH}_{{\bf B_0}}$, which we call the fiber strong shape functor. By definition,
\[\Ss_{{\rm B_0} }(X)=X, X\in ob({\bf HTop}_{{\bf B_0}})\]
and
\[\Ss_{{\rm B_0} }([\varphi]_{{\rm B_0} })=\Ss^{'}(\varphi), [\varphi]_{{\rm B_0} }\in {\rm {Mor}}_{{\bf HTop}_{{\bf B_0}}}(X,Y).\]

Let us define a functor $\SSS:{\bf SSH}_{{\bf B_0}}\to \bsh_{{\bf B_0}}$. Assume that $\SSS(X)=X$ for each object $X\in ob({\bf SSH}_{{\bf B_0}})$. Let $F:(X,\pi_{X})\to (Y,\pi_{Y})$ be a fiber strong shape morphism given by a triple ($\bp, \bq, [f]$).

Consider the morphism $\E([f])$ as an image of $[f]$ with respect the functor $\E:{\bf CPH}{\bf Top}_{{\bf B_0}}\to \bpro-{\bf HTop}_{{\bf B_0}}$. The triple ($\HH\bp, \HH\bq, \E[f]$) generates a fiber shape morphism, which we denote by $\SSS(F):(X,\pi_{X})\to (Y,\pi_{Y})$.

Now we can formulate the following 
\begin{theorem}\label{theorem 2.3}
There exists a commutative diagram
\begin{center}
 \begin{tikzpicture}
\node (A) {};
\node (B) [node distance=6cm, right of=A] {$\bsh_{{\bf B_0}}$};
\node (C) [node distance=1cm, below of=A] {};
\node (D) [node distance=3cm, right of=C] {${\bf HTop}_{{\bf B_0}}$};
\node (E) [node distance=1cm, below of=C] {};
\node (F) [node distance=6cm, right of=E] {${\bf SSH}_{{\bf B_0}},$};
\draw[->] (D) to node [above left]{$\SSS_{{\rm B_0} }$} (B);
\draw[->] (F) to node [right]{$\SSS$} (B);
\draw[->] (D) to node [below left]{$\Ss_{{\rm B_0} }$} (F);
\end{tikzpicture}
 \end{center}
where $\SSS_{{\rm B_0} }$ is V.Baladze fiber shape functor $\Bci$.\qed
\end{theorem}
\begin{corollary}\label{corollary}
Let $(X,\pi_{X})$ and $(Y,\pi_{Y})$ be topological spaces over $\rm B_0 $. If $\ssshh_{{\rm B_0} }(X)=\ssshh_{{\rm B_0} }(Y)$, then $\sh_{{\rm B_0} }(X)=\sh_{{\rm B_0} }(Y)$.\qed
\end{corollary}

\begin{remark}\label{remark 2.7}
Using the methods developed in this paper and papers ([B$_6$], [L-M],[M$_1$], [M$_2$]) it is possible to construct fiber strong shape theory for category of arbitrary continuous maps.
\end{remark}

\end{document}